\documentclass[12pt]{amsart}

\usepackage{amssymb,amsfonts,amsmath}
\usepackage{latexsym}
\usepackage{graphicx}
\usepackage{epsf}

\newtheorem{theorem}{Theorem}
\newtheorem{coro}{Corollary}
\newtheorem{lemma}{Lemma}
\newtheorem{conj}{Conjecture}
\newtheorem{prop}{Proposition}

\newtheorem{defi}{Definition}

\renewcommand{\geq}{\geqslant}
\renewcommand{\leq}{\leqslant}

\newcommand{\R}{\mathbb{R}}
\newcommand{\E}{\mathbb{E}}

\newcommand{\eps}{\varepsilon}

\newcommand{\la}{\langle}
\newcommand{\ra}{\rangle}

\begin{document}

\title{Stability results for the volume of random  simplices }

\author{Gergely Ambrus}
\email[G. Ambrus]{ambrus@renyi.hu}
\thanks{The research of the first named author was supported by OTKA grants 75016 and
76099.}

\author{K\'aroly J. B\"or\"oczky}
\email[K. B\"or\"oczky]{carlos@renyi.hu}
\thanks{ The second named author was supported by OTKA grants 068398 and 75016, and by the EU Marie
Curie IEF project GEOSUMSET}

\address{R\'enyi Institute of the Hungarian Academy of Sciences, PO Box 127, 1364 Budapest, Hungary}


\maketitle

\begin{abstract}
It is known that for a convex body $K$ in $\R^d$ of volume one, the expected
volume of random simplices in $K$ is minimised if $K$ is an ellipsoid, and for
$d=2$, maximised if $K$ is a triangle. Here we provide corresponding stability
estimates.
\end{abstract}

\section{Introduction and history}

Let $K$ be a convex body in $\R^d$. What is the expected value of the volume of
a random simplex in $K$? Naturally, this question needs to be clarified
further. We will work with two (or three) models: in the first, all the
vertices of the simplex are chosen uniformly and independently from $K$, while
in the second, one vertex is at a fixed position -- in a special case, this is
the centroid of $K$. We are interested in other moments as well, and also, we
would like the answer to be invariant under affine transformations.

As a general reference for stochastic geometry, we refer to R. Schneider, W.
Weil \cite{ScW08}, and for convexity, to T. Bonnesen, W. Fenchel \cite{BoF34},
P.M. Gruber \cite{Gru07} and R. Schneider \cite{Sch93}. $V$~or $V_d$ stands for
the $d$-dimensional volume (if the dimension is clear, we shall omit $d$), the
convex hull of the points $x_1 \dots, x_n$  is denoted by $[x_1,\ldots,x_n]$,
and $\gamma(K)$ is the centroid of $K$.

\begin{defi}\label{edef}
Let $K$ be a convex body in $\R^d$. For any $n \geq d+1$ and $p>0$, let
\[
\E^p_n(K)=V(K)^{-n-p}\int_{K}\ldots\int_{K}V([x_1,\ldots,x_n])^p \,dx_1\ldots
dx_n.
\]
Further, for a fixed $x \in \R^d$, let
\[
\E^p_x(K)=V(K)^{-d-p}\int_{K}\ldots\int_{K}V([x,x_1,\ldots,x_d])^p \,dx_1\ldots
dx_d.
\]
Specifically, we write $\E^p_*(K)$ for $\E^p_x(K)$, when $x = \gamma(K)$.
\end{defi}
In particular, for integer $p$, $\E^p_{d+1}(K)$ is the expectation of the
$p\,$th moment of the relative volume of simplices in $K$. Clearly, $\E^p_n(K)$
and  $\E^p_*(K)$ are invariant under non-singular affine transformations, and $\E^p_o(K)$ is
invariant under non-singular linear transformations, where $o$ stands for the
origin. We note that for fixed $K$ and $p\geq 1$, $\E^p_x(K)$ is a strictly convex
function of $x$, therefore it attains its minimum at a unique point. If $K$ is
$o$-symmetric, then the minimum is attained at $o$, and $\E^p_o(K) =
\E^p_*(K)$.

In the rest of the section, we give an overview of the history of the
quantities introduced in Definition~\ref{edef} and their various connections.
The main results are presented in Section~\ref{results}, whose proofs are found
in the subsequent parts. Section~\ref{pettystab} contains further corollaries.

\bigskip
\noindent {\bf 1.1.} {\em Sylvester's problem.} The quantity $\E^p_{d+1}(K)$
arose right at the first steps of random convex geometry. Indeed, (probably)
the first question in this topic is due to Sylvester \cite{Syl64}: in 1864 he
(vaguely) asked, what is the probability that four randomly chosen points in a
planar convex disc are in convex position, that is, none of them is in the
convex hull of the other three. Generalising to higher dimensions, if  $d+2$
points are chosen randomly from a convex body $K \subset \R^d$, then the sought
quantity is exactly $1-(d+2) \E^1_{d+1}(K)$. It is then natural to ask: for
which convex bodies is this probability minimal and maximal? The first  steps
in this direction were taken by Blaschke (\cite{Bla17} and \cite{Bla23}), who
showed that in the plane, the probability in question is maximal for ellipses,
and minimal for triangles. The maximisers in higher dimensions are the
ellipsoids (cf. Groemer \cite{Gro73}), whereas the minimiser bodies in higher
dimensions are still not known. We shall state these results as theorems later.
For a thorough historical account of this problem, see Klee \cite{Kle69}, and
also B\'ar\'any \cite{Bar08}.

\bigskip
\noindent {\bf 1.2.} {\em Minimisers and affine inequalities.} Let $K\subset
\R^d$ be a convex body with $\gamma(K)=0$. The {\em intersection body} $I K$ of
$K$ is defined by its radial function:
\[
\rho_{IK}(u) = V_{d-1}(K \cap u^\perp).
\]
H. Busemann \cite{Bus53} established the formula
\begin{equation}\label{buseformula}
V_d(K)^{d-1} = \frac{(d-1)!}{2} \int_{S^{d-1}} V_{d-1}(K \cap u^\perp)^d \,
\E^1_o(K \cap u^\perp) d \sigma(u),
\end{equation}
where $\sigma$ is surface area measure on $S^{d-1}$. In the same paper, he
proved the {\em Busemann random simplex inequality}:
\begin{equation}\label{busemannrandom}
\E^1_o(K) \geq \E^1_o(B^d).
\end{equation}
Combining \eqref{buseformula} and \eqref{busemannrandom}, he derived the {\em
Busemann intersection inequality}, stating that the volume of the intersection
body is maximal for ellipsoids:
\begin{equation}\label{busemanninter}
V_d(I K) \leq \frac{\kappa_{d-1}^d}{\kappa_d^{d-2}} V(K)^{d-1},
\end{equation}
where $\kappa_d = V_d (B^d)$.

A couple of years later, Petty \cite{Pet61} introduced {\em centroid
bodies}$\,$: the centroid body $\Gamma K$ of $K$ is the convex body  in $\R^d$
defined by the support function
\[
h_{\Gamma K} (u) = \frac{1}{V(K)} \int_K |\la u,x \ra| dx.
\]
Using an approximation argument and the volume formula for zonotopes, he
obtained the following formula for the volume of $\Gamma K$:
\begin{equation}\label{centroidbody}
V_d(\Gamma K) = 2^d V_d(K) \E^1_o(K).
\end{equation}
The argument is nicely presented in \cite{Gar06}. Using the Busemann random
simplex inequality \eqref{busemannrandom},  Petty obtained the {\em
Busemann-Petty centroid inequality}, which states that the volume of the
centroid body is minimal for ellipsoids:
\begin{equation}\label{busepetty}
\frac{V_d(\Gamma K)}{V_d(K)} \geq \left( \frac{2 \kappa_{d-1}}{(d+1) \kappa_d}
\right)^d.
\end{equation}
The conjectured converse of this inequality is that the volume is maximised for
simplices provided that $o$ is the centroid; this would be crucial in high
dimensional convex geometry, as we shall soon see.

The minimisers of the mean volumes of random simplices are known in full
generality: they are the ellipsoids for all the quantities introduced in
Definition~\ref{edef}.

\begin{theorem}[Blaschke, Busemann, Groemer]
\label{ellipsoid} For any convex body $K$ in $\R^d$, for any $p\geq 1$, and for
any $n \geq d+1$, we have
$$
\mbox{$\E^p_o(K)\geq\E^p_o(B^d)$ \ and \ $\E^p_*(K)\geq\E^p_*(B^d)$ \ and \
$\E^p_n(K)\geq\E^p_n(B^d)$.}
$$
Here $\E^p_o(K)=\E^p_o(B^d)$ if and only if $K$ is an $o$-symmetric ellipsoid,
and $\E^p_*(K)=\E^p_*(B^d)$ or $\E^p_n(K)=\E^p_n(B^d)$ if and only if $K$ is an
ellipsoid.
\end{theorem}

As we noted before, Blaschke \cite{Bla23} handled $\E^1_3(K)$ in the planar
case, Groemer \cite{Gro73} extended his result to higher dimensions,  H.
Busemann \cite{Bus53}) obtained the estimate for $\E^p_o(K)$. Groemer
\cite{Gro74} derived the result for $\E^p_n(K)$. All the proofs are similar and
based on Steiner symmetrisation. For thorough discussions of these inequalities
and relatives, see the survey article \cite{Lut93} by E. Lutwak, or the
monograph \cite{Gar06} by R.J. Gardner. The minimal values
in the cases of random simplices when $p\geq 1$ is an integer, can be
found as Theorems~8.2.2 and 8.2.3 in  R. Schneider, W. Weil \cite{ScW08}.
Writing $\kappa_d=V(B^d)=\frac{\pi^{d/2}}{\Gamma(\frac{d}2+1)}$, we have
\begin{eqnarray*}
\E^p_*(B^d)&=&{d+p \choose d}^{-1}\kappa_d^{-d-p}\kappa^{d}_{d+p}\cdot
\frac{\kappa_1\ldots\kappa_d}{\kappa_{p+1}\ldots\kappa_{p+d}}\\
\E^p_{d+1}(B^d)&=&(d!)^{-p}{d+p \choose d}^{-1}\kappa_d^{-d-p-1}\kappa^{d+1}_{d+p}\cdot
\frac{\kappa_{d^2+dp+d}}{\kappa_{d^2+dp+d+p}}\cdot
\frac{\kappa_1\ldots\kappa_d}{\kappa_{p+1}\ldots\kappa_{p+d}}
\end{eqnarray*}

\bigskip
\noindent {\bf 1.3.} {\em Maximum inequalities and the slicing conjecture.} As
usual, let $K$ be a convex body in $\R^d$, and assume that $\gamma(K)=o$. The
{\em inertia matrix} of $K$ is the $d \times d$ matrix $M$ given by
\[
M_{ij} =\int_K x_i x_j \, dx,
\]
where $x_i$ is the $i$th coordinate of $x$. Since for any $y \in \R^d$, we have
$y^\top M y = \int_K \la x,y \ra^2 dx$, it follows that $M$ is a positive
definite, symmetric matrix, and hence it has a positive square-root $A$. The
inertia matrix of the  convex body $A^{-1} K$ is then $I_d / \det A$ (see
J. Bourgain, M. Meyer, V. Milman, A. Pajor \cite{BMMP}, and for a more detailed
discussion, see Ball~\cite{Bal88}).
For a non-singular affine transformation $\Phi \in GL_d$,
we say that $\Phi K$ is in {\em isotropic position} with the constant of
isotropy $L_K$, if $\gamma(K)=o$, $V_d(\Phi K) =1$, and the inertia matrix of
$\Phi K$ is a multiple of the identity, that is,
\[
\int_{\Phi K} \la x,y \ra^2 dx = L_K^2 \|y \|^2
\]
for every $y \in \R^d$.  We just have seen that every convex body has a
non-singular affine image that is in isotropic position, and it is well known
that the isotropic position is unique up to orthogonal transformations. Hence,
$L_K$ is an affine invariant. Moreover,
\begin{equation}\label{lkm}
L_K = (\det M)^{1/2d} V_d(K)^{-(d+2)/2d}.
\end{equation}
By expanding the determinant of $M$, one obtains (see Blaschke~\cite{Bla18} or
Giannopoulos~\cite{Gian})
\[
\det M =d! \, V_d(K)^{d+2} \E^2_* (K),
\]
and hence from (\ref{lkm}),
\begin{equation}\label{lke2}
L_K^{2d} = d! \, \E^2_* (K).
\end{equation}
The {\em slicing conjecture}, initiated by J. Bourgain \cite{Bou86},
asserts that there exists a universal constant
$L$, for which $L_K \leq L$ for every convex body $K$, regardless of the
dimension. There are various equivalent formulations of this major open
problem; for thorough surveys, consult the papers  V.D. Milman and A. Pajor
\cite{MiP89}, and A.A. Giannopoulos and V.D. Milman
\cite{GiM04} for some later results.

By (\ref{lke2}), one has to determine the maximum of $\E^2_*(K)$. The most
general conjecture is the following, where $T^d$ stands for a $d$-dimensional
simplex:

\begin{conj}[Simplex conjecture]
\label{polysimplexconj} If $K$ is a convex body in $\R^d$, then for any $p\geq
1$ and for any $n \geq d+1$,
$$
\E^p_*(K)\leq\E^p_*(T^d) \textrm{ and } \E^p_n(K)\leq\E^p_n(T^d),
$$
with equality if and only if $K$ is a simplex.
\end{conj}

Little is known about Conjecture~\ref{polysimplexconj}. The proposed extremal
values  are known explicitly only in a few cases.
 W.J. Reed \cite{Ree74}
proved that if $p\geq 1$ is an integer, then
\begin{multline*}
\E^p_3(T^2)\\=\frac{12}{(p+1)^3(p+2)^3(p+3)(2p+5)}
\left[6(p+1)^2+(p+2)^2\sum_{i=0}^p{p \choose i}^{-2}\right].
\end{multline*}
 For all $n$,
only the first moments $\E^1_n(T^2)$ and
 $\E^1_n(T^3)$ are known, see C. Buchta \cite{Buc84}
and C. Buchta, M. Reitzner \cite{BuR01}, respectively. Even explicit values of
$\E^1_{d+1}(T^d)$ for $d\geq 4$ are missing. It is important that for any $d
\geq 2$,
\begin{equation}\label{e2simplex}
\E^2_*(T^d) \leq \frac{1}{d!},
\end{equation}
see Giannopoulos~\cite{Gian}. Thus, the simplex conjecture for $\E^2_*(K)$
implies the slicing conjecture.

The method of Dalla and Larman \cite{DaL91}, who considered $\E^1_n(K)$,
combined with Theorem~\ref{CCG} of Campi, Colesanti, and Gronchi \cite{CCG99}
yields Conjecture~\ref{polysimplexconj} if $K$ is a polytope of at most $d+2$
vertices. B\'ar\'any and  Buchta \cite{BaB93} proved the following asymptotic
version of  Conjecture~\ref{polysimplexconj}  for $p=1$. If $K$ is not a
simplex, there exists a threshold $n_K$ depending on $K$, such that
$\E^1_n(K)<\E^1_n(T^d)$ for $n>n_K$. Conjecture~\ref{polysimplexconj} for all
$K$ and $n$ is verified only in the plane.

\begin{theorem}[Blaschke,Dalla-Larman,Giannopoulos]
\label{polytrian} If $K$ is a planar convex body, then for any $n\geq 3$ and $p
\geq 1$,  $ \E^p_n(K)\leq\E^p_n(T^2)$, with equality if and only if $K$ is a
triangle.
\end{theorem}

More precisely, it was proved by  Blaschke \cite{Bla17} for $n=3$, and by Dalla
and Larman \cite{DaL91} for $n\geq 4$, that $\E^p_n(K)\leq\E^p_n(T^2)$. In
addition, Giannopoulos \cite{Gia92} verified that equality holds only if $K$
itself is a triangle.

We shall see in Section~\ref{secplane}  (compare \eqref{basicmax} and
Lemma~\ref{trianglehexagon})
 that the method of S. Campi, A. Colesanti, P. Gronchi
\cite{CCG99}, see Theorem~\ref{CCG}, leads to the planar version of the first
statement of Conjecture ~\ref{polysimplexconj}.

\begin{theorem}
\label{triangleoptslice}
If $K$ is a  convex disc, then for any $p \geq 1$, we
have $\E^p_*(K)\leq\E^p_*(T^2)$, with equality if and only if $K$ is a
triangle.
\end{theorem}

For centrally symmetric planar convex discs and $p=1$, T. Bisztriczky and  K. B\"or\"oczky
Jr. \cite{BiB01} proved the analogue of Theorem~\ref{triangleoptslice} with
$o$-symmetric parallelograms instead of triangles as maximisers.
The method readily extends to all $p\geq 1$.

\bigskip
\noindent{\bf 1.4.} {\em Equivalence.} Finally, we establish connections
between the different quantities measuring the mean volumes of random
simplices.

For every $p \geq 1$ and for any convex body $K \subset \R^d$, we have
\begin{equation}\label{p*}
(\E^p_*(K))^{1/p} \leq (\E^p_{d+1}(K))^{1/p} \leq (d+1)(\E^p_*(K))^{1/p}.
\end{equation}
For a proof, see Proposition 1.3.1 of Giannopoulos~\cite{Gian}.

Specifically, for $p=2$, one obtains
\begin{equation}\label{2*}
 (d+1)\E^2_*(K)=\E^2_{d+1}(K).
\end{equation}
The proof goes by assuming that $\gamma(K)=o$ and $K$ is in isotropic position.
Given $x_1, \dots , x_{d+1} \in \R^d$,
\[
V_d([x_1, \dots, x_{d+1}]) = \frac{1}{d!}\det((x_1, 1), \dots, (x_{d+1},1)).
\]
Using this formula and proceeding as in Proposition 3.7. of Milman and
Pajor~\cite{MiP89}, one obtains \eqref{2*}.

Thus, in view of \eqref{lke2}, to prove the slicing conjecture, it would
suffice to estimate $\E^2_{d+1}(K)$.

Next, we show that all the quantities $\E^p_*(K)$ and $\E^p_{d+1}(K)$ are
equivalent in the following sense: for any $p,q>0$, there exist constants
$c_{p,q}$ and $C_{p,q}$ depending on $p$ and $q$ only, such that if $\E^p (K)$
stands for either $\E^p_*(K)$ or $\E^p_{d+1}(K)$, then
\begin{equation}\label{equi}
c_{p,q}^d(\E^p(K))^{1/p} \leq  (\E^q(K))^{1/q} \leq C_{p,q}^d (\E^p(K))^{1/p}.
\end{equation}
To this end, using \eqref{p*}, it suffices to show that $\E^p_*(K)$ and
$\E^q_*(K)$ are equivalent.  H\"older's inequality implies that for $0<p<q$,
\begin{equation}\label{holderpq}
(\E^p_*(K))^{1/p} \leq (\E^q_*(K))^{1/q}.
\end{equation}

To see the estimate in the other direction, we refer to Milman and Pajor
\cite{MiP89}. Proposition 3.7 therein states that there exists an absolute
constant $c>0$, such that for any convex body $K \subset \R^d$, and for any $0
<p \leq 2$,
\begin{equation}\label{equi2}
(\E^2_*(K))^{1/2} \leq c^d (\E^p_*(K))^{1/p}.
\end{equation}

 The key step is using the concentration of volume property of
convex bodies (indeed, for log-concave functions), cf. Borell's lemma, which
then establishes that for a fixed $v \in \R^d$, all the $L_p$-norms $(\int_K
|\la x, v \ra|^p dx)^{1/p}$  are equivalent. Then, one uses the fact that
fixing $x_1, \dots, x_{d-1}$, $V[x_1, \dots, x_d]$ is a linear function of
$x_d$, and hence,
\[
\E^p_{d+1}(K) = \int_K |\la x_d, v \ra|^p dx_d
\]
for some $v \in \R^d$, provided $V_d(K)=1$. Equation \eqref{equi2} can then be
obtained by an inductive argument, provided  $K$ is in isotropic position.

When $p>2$, then we use the following Khinchine type inequality: if $K \subset
\R^d$ is a convex body of volume 1, then for any $v \in \R^d$,
\[
\left(\int_K |\la x, v \ra|^p dx \right)^{1/p} \leq c p \int_K |\la x, v \ra|\,
dx \leq c p \left(\int_K |\la x, v \ra|^2 dx \right)^{1/2}
\]
for some universal constant $c$ (see Proposition 2.1.1. of
Giannopoulos~\cite{Gian}). Then the argument of Milman and Pajor works,
yielding that there exists a constant $C$, such that
\[
\left( \frac{C}{p}\right)^d (\E^p_*(K))^{1/p} \leq (\E^2_*(K))^{1/2}.
\]
Referring to \eqref{holderpq} and \eqref{equi2}, we arrive to \eqref{equi}.

We note that in order to prove the slicing conjecture, using formulas
\eqref{lke2} and \eqref{equi}, it would suffice to verify either the first
 or the second statement (with $n=d+1$) of Conjecture~\ref{polysimplexconj}
 for any particular $p \geq 1$.

\section{Main results}\label{results}

Our goal is to provide stability versions of Theorems~\ref{ellipsoid},
\ref{polytrian} and \ref{triangleoptslice}. We shall use the  {\em Banach-Mazur
distance} $\delta_{\rm BM}(K,M)$ of the convex bodies $K$ and $M$, which is
defined by
\begin{multline*}
\delta_{\rm BM}(K,M)=\min\{\lambda\geq 1:\,
K-x\subset \Phi(M-y)\subset \lambda(K-x)\\
\mbox{ for \ } \Phi\in{\rm GL}_d,\,x,y\in\R^d\}.
\end{multline*}
If $K$ and $M$ are $o$-symmetric, then $x=y=o$ can be assumed.  It follows by
Fritz John's ellipsoid theorem that $\delta_{\rm BM}(K,B^d)\leq d$ for any
$d$-dimensional convex body $K$, and $\delta_{\rm BM}(K,B^d)\leq \sqrt{d}$
holds if $K$ is centrally symmetric. Moreover, J.~Lagarias and G. Ziegler
verified in \cite{LaZ91} that $\delta_{\rm BM}(K,T^d)\leq d+2$.

First, the stability version of Theorem~\ref{ellipsoid}.

\begin{theorem}\label{ellipsoidstab}
If $K$ is a convex body in $\R^d$ with $\delta_{\rm BM}(K,B^d)=1+\delta$
for $\delta>0$, then
for any $p \geq 1$,
\begin{align*}
\E^p_*(K)&\geq(1+\gamma^p\delta^{d+3}) \E^p_o(B^d)\\
\E^p_{d+1}(K)&\geq(1+\gamma^p\delta^{d+3}) \E^p_{d+1}(B^d),
\end{align*}
where the constant $\gamma>0$ depends on $d$ only. Moreover, if $K$ is
centrally symmetric, then the error terms can be replaced by $\gamma^p
\delta^{(d+3)/2}$.
\end{theorem}

Similar stability estimates preceded our work. Groemer \cite{Gro94} showed that
under rather strict regularity conditions on the  boundary of $K$, the above
statement holds with an error term of order $\delta^{c\,d^2}$ for some universal
constant $c$. Fleury, Gu\'edon and Paouris \cite{FGP07} proved a stability result
for the mean width of $L_p$-centroid bodies, which in the case $p=1$, yields a
stability estimate for $\E^1_o(K)$ by \eqref{centroidbody}. However, the error
term obtained this way is again only of order $\delta^{c\,d^2}$ for some
universal constant~$c$. We remark that for $p\neq 1$, no such direct connection
exists between $\E^p_o(K)$ and the volume of the $L_p$-centroid body.

Second, the stability version of Theorems~\ref{polytrian} and
\ref{triangleoptslice}.

\begin{theorem}
\label{triangleoptstab} If $K$ is a planar convex body with $\delta_{\rm
BM}(K,T^2)= 1+ \delta$ for some $\delta>0$, and $p \geq 1$, then
\begin{eqnarray*}
\E^p_*(K)&\leq&(1-c^p\delta^2) \E^p_*(T^2)\\
\E^p_3(K)&\leq&(1-c^p\delta^2) \E^p_3(T^2),
\end{eqnarray*}
where $c$ is a positive absolute constant. This estimate is
asymptotically sharp as $\delta$ tends to zero.
\end{theorem}

\section{Linear shadow systems and Steiner symmetrisation}
\label{secshadow}

 For obtaining the stability versions of both the minimum
and maximum inequalities, we shall use the following notion. Given a compact
set $\Xi$ in $\R^d$, a unit vector $v$, and for each $ x \in\Xi$, a speed
$\varphi(x)\in\R$, the corresponding {\em shadow system} is
$$
\Xi_t=\{ x+t\varphi(x)v:x\in\Xi \} \mbox{ \ for $t\in\R$}.
$$
According to the classical work of H. Hadwiger \cite{Had57}, C.A. Rogers, G.C.
Shephard \cite{RoS58} and Shephard \cite{She64},

\begin{theorem}[Hadwiger,Rogers,Shephard]
\label{shadow} For a shadow system $\Xi_t$, every quermassintegral of $\Xi_t$
is a convex function of $t$.
\end{theorem}
We note that  for any $p\geq 1$, the convexity of the $p$th moment of the
quermassintegrals follows as well.

In the last decades, shadow systems were successfully applied to various
extremal problems about convex bodies (see e.g. S. Campi, P. Gronchi
\cite{CaG06}, and M. Meyer, Sh. Reisner \cite{MeR06}). For our purposes, we
need a restricted class of shadow movements, introduced in \cite{CCG99} by S.
Campi, A. Colesanti, and P. Gronchi. We say that $K_t$, $t\in[a,b]$, is a {\em
linear shadow system of convex bodies}, if we start with a convex body $K$,
 the speed $\varphi(x)$ is constant along any chord of $K$ parallel to
 $v$,
 and
$$
K_t=\{x+t\varphi(x)v:\,x\in K\} \mbox{ \ for $t\in[a,b]$}
$$
is convex for every $t \in [a,b]$. In this case, $\varphi(x)$ is continuous on
$K$, and it depends only on the projection $\pi_vx$ of $x$ to $v^\bot$. Moreover, the
volume of $K_t$ is constant, and the transformation $x\mapsto x+t\varphi(x)v$
from $K$ to $K_t$ is measure preserving.

For any linear shadow system $K_t$, there also exists
 a linear shadow system $\widetilde{K}_t$, $t\in[a,b]$,
 such that
\begin{equation}
\label{newshadow} \mbox{$\gamma(\widetilde{K}_t)=o$ for $t\in[a,b]$, and each
$\widetilde{K}_t$ is a translate of $K_t$.}
\end{equation}
To see this,  note that
\begin{equation}
\label{shadowcentroid}
 \gamma(K_t)=\gamma(K)+t\cdot v\cdot V(K)^{-1}\int_{K}\varphi(z)\,dz.
\end{equation}
Therefore, $\widetilde{K}_t=K_t-\gamma(K_t)$ can be achieved by using the speed
$$
\tilde{\varphi}(x)=\varphi(x+\gamma(K))-V(K)^{-1}\int_{K}\varphi(z)\,dz \mbox{
\ for $x\in \widetilde{K}$}.
$$

The main reason for restricting shadow movements is the following result of
\cite{CCG99} (where linear shadow systems were called RS-movements).

\begin{theorem}[Campi, Colesanti, Gronchi]\label{CCG}
If $K_t$, $t\in[a,b]$, is a linear shadow system, then $\E^p_n(K_t)$,
$\E^p_o(K_t)$ and $\E^p_*(K_t)$ are convex functions of $t$. If either of these
convex functions is linear, then any two elements of the system are affine
images of each other, and actually linear images in the case of $\E^p_o(K_t)$.
\end{theorem}

We note that although Theorem~\ref{CCG} was proved only for $\E^p_n(K_t)$ in
\cite{CCG99}, the method works for the other functionals as well
(see also Lemma~1 for a direct approach). Indeed, for
handling $\E^p_n(K_t)$, the authors consider for each $n$-tuple
$\Xi=\{x_1,\ldots,x_n\}\subset K$ the associated shadow system
$$
\Xi_t=[x_1+t\varphi(x_1)v,\ldots,x_n+t\varphi(x_n)v].
$$
Since $V_d(\Xi_t)$ is a convex function of $t$ by Theorem~\ref{shadow},
we conclude Theorem~\ref{CCG} by
\begin{eqnarray*}
\E^p_n(K_t)&=&V(K)^{-n-p}\times\\
&&\int_{K}\ldots\int_{K}
V([x_1+t\varphi(x_1)v,\ldots,x_n+t\varphi(x_n)v])^p \,dx_1\ldots dx_n.
\end{eqnarray*}
 In order to
obtain the convexity of $\E^p_o(K_t)$, to each $d$-tuple
$\{x_1,\ldots,x_d\}\subset K\backslash o$ one assigns the $d+1$-tuple
$\Xi=\{o,x_1,\ldots,x_d\}$, and defines the speed of $o$ to be zero. The
convexity $\E^p_*(K_t)$ follows from \eqref{newshadow}.

Finally, we have to
deal with the extremal situations only. The argument is based on ideas in
\cite{CCG99}. Let us indicate it in the  case when $\E^p_o(K_t)$ is a linear
function of $t$, which also settles the case when $\E^p_*(K_t)$ is a linear
function of $t$. If for some $s,t\in[a,b]$, $s<t$, $K_t$ and $K_s$
 are not images of each other by any linear transformation, then
there exist $\tau+\mu,\tau-\mu\in[s,t]$, $\mu>0$,
and $d$-tuple $\{x_1,\ldots,x_d\}\subset K$
with the property that
 $\{x_1+\tau\varphi(x_1)v,\ldots,x_d+\tau\varphi(x_d)v\}$ is
 linearly dependent, and
$\{x_1+(\tau+\mu)\varphi(x_1)v,\ldots,x_d+(\tau+\mu)\varphi(x_d)v\}$
 is linearly independent.
It follows for $\Xi=\{x_1,\ldots,x_d,o\}$ that
$V_d(\Xi_{\tau})<\frac12(V_d(\Xi_{\tau-\mu})+V_d(\Xi_{\tau+\mu}))$,
which in turn yields
$\E^p_o(K_{\tau})<\frac12(\E^p_o(K_{\tau-\mu})+\E^p_o(K_{\tau+\mu}))$
by Theorem~\ref{shadow} and the continuity of $\varphi$.

When dealing with linear shadow systems, the following simple observation is
very useful. If $p>0$, $\sigma_0,\ldots,\sigma_d$ are parallel segments, and $\Phi$ is
an affine transformation that acts by translation along any line parallel to
the $\sigma_i$'s, then
\begin{equation}
\label{translateo}
\begin{split}
\int_{\sigma_1}\ldots\int_{\sigma_d}
&V([o,z_1,\ldots,z_d])^p\,dz_1\ldots dz_d\\
&= \int_{\Phi\sigma_1}\ldots\int_{\Phi\sigma_d} V([\Phi o,\Phi z_1,\ldots,\Phi
z_d])^p\,dz_1\ldots dz_d,
\end{split}
\end{equation}
and
\begin{equation}
\label{translateall} \begin{split} \int_{\sigma_0}\ldots\int_{\sigma_d}
&V([z_0,\ldots,z_d])^p\,dz_0\ldots dz_d \\
&=\int_{\Phi\sigma_0}\ldots\int_{\Phi\sigma_d} V([\Phi z_0,\ldots,\Phi
z_d])^p\,dz_0\ldots dz_d.
\end{split}
\end{equation}

All the known proofs of Theorem~\ref{ellipsoid} use the fact that the moments
to be estimated are monotone decreasing with respect to Steiner symmetrisation.
This is a consequence Theorem~\ref{CCG}, due to the following connection
between Steiner symmetrals and shadow systems. Let $K$ be a convex body, and
$H$ a hyperplane. Consider the unique linear shadow system $K_t$, $t\in[-1,1]$,
such that $K_1=K$, and $K_{-1}$ is the reflected image of $K$ through $H$. Then
$K_0$ is the Steiner symmetral $K_H$ of $K$ with respect to $H$. Now,
Theorem~\ref{ellipsoid} follows by using the well-known fact that
$V(K)^{\frac{1}d}B^d$ can be obtained as a limit of a sequence of Steiner
symmetrals starting from $K$.

The behaviour of $\E^p_{d+1}(K)$, $\E^p_o(K)$ and $\E^p_*(K)$ under Steiner
symmetrisation can be computed easily using basic properties of determinants.
Refining the proof, we will be able to deduce the stability estimates. It goes
as follows. Assume that we take the Steiner symmetral of $K$ with respect to
$H$. Let $x_0, \dots, x_d$ be an arbitrary set of points of $H$, and consider
the integral over those simplices whose vertices project to the points $(x_i)$
in $H$. By \eqref{translateo} and \eqref{translateall}, we may assume that the
midpoints of the chords of $K$ through $x_0, \dots, x_{d-1}$ are located in
$H$. Then the Steiner symmetrisation moves only $\sigma(x_d)$, and the
situation is easily handled.

For Lemmas~\ref{steiner} and \ref{steinerstab}, let $x_0,\ldots,x_d$ be
contained in a hyperplane $H$ in $\R^d$ in a way such that no $d$ of them are
contained in any $(d-2)$-plane, and let $v$ be a unit vector not parallel to
$H$. In addition, let $\delta>0$, $\alpha_0\geq 0$, and $\alpha_i>0$ for
$i=1,\ldots,d$. For Lemma~\ref{steiner}, to save space, we also use the
(slightly obscure) convention that $\int_{J_0}dt_0=1$ for $J_0 = \{x_0\}$.

\begin{lemma}
\label{steiner} Let $p\geq 1$, let $0\leq\beta_i<\alpha_i$ for $i=1,\ldots,d$,
and let $\beta_0=\alpha_0$, if $\alpha_0=0$, and  $0\leq\beta_0<\alpha_0$ if
$\alpha_0>0$. For $J_i=[-\alpha_i,-\beta_i]\cup[\beta_i,\alpha_i]$,
$0=1,\ldots,d$, we have
$$
\varphi(s)=\int_{J_d\,+s} \int_{J_{d-1}}\ldots \int_{J_0}
V([x_0+t_0v,\ldots,x_d+t_dv])^p\,dt_0\ldots dt_d
$$
is convex, and $\varphi(s)\geq \varphi(0)$.
\end{lemma}

\begin{proof} For any fixed $t_i\in J_i$, $i=0,\ldots,d$, the function
$$
V([x_0+t_0v,\ldots,x_{d-1}+t_{d-1}v,x_d+(t_d+s)v])^p
$$
of $s$ is convex because it is the $p$th power of the absolute value of a
linear function. Therefore $\varphi(s)$ is convex as well. Since $\varphi(s)$
is even, we have $\varphi(s)\geq\varphi(0)$.
\end{proof}

Naturally, Lemma~\ref{steiner} with $\beta_i=0$, $i=0,\ldots,d$, directly
yields Theorem~\ref{CCG} for $\E^p_{d+1}(K)$, $\E^p_o(K)$ and $\E^p_*(K)$. Now
we provide a stability version under a technical (but necessary) side condition.

\begin{lemma}
\label{steinerstab}
Let $p \geq 1$ and $\delta \in(0,\alpha_d/2)$, and assume that if
$|t_i|\leq\alpha_i$ for every $i=0,\ldots,d-1$, then
\begin{equation}
\label{intcond} {\rm aff}\{x_0+t_0v,\ldots,x_{d-1}+t_{d-1}v\}\cap
[x_d-(\alpha_d-\delta)v,x_d+(\alpha_d-\delta)v]\neq\emptyset.
\end{equation}
Then the following inequalities hold.
\begin{itemize}
\item[(i)] In the case $\alpha_0=0$:
\begin{align*}
&\int_{-\alpha_d+\delta}^{\alpha_d+\delta}
\int_{-\alpha_{d-1}}^{\alpha_{d-1}}\ldots \int_{-\alpha_1}^{\alpha_1}
V([x_0,x_1+t_1v,\ldots,x_d+t_dv])^p\,dt_1\ldots dt_d\\
&\qquad- \int_{-\alpha_d}^{\alpha_d} \int_{-\alpha_{d-1}}^{\alpha_{d-1}}\ldots
\int_{-\alpha_1}^{\alpha_1} V([x_0,x_1+t_1v,\ldots,x_d+t_dv])^p\,dt_1\ldots
dt_d\\
&\geq\delta^2\,\frac{p2^{d-p-1}}{d^p}\,\alpha_1\ldots\alpha_{d-1}\alpha_d^{p-1}\,
V_{d-1}(\pi_v[x_0,\ldots,x_{d-1}])^p.
\end{align*}
\item[(ii)] If $\alpha_0>0$, then
\begin{align*}
&\int_{\delta-\alpha_d}^{\delta+\alpha_d}
\int_{-\alpha_{d-1}}^{\alpha_{d-1}}\ldots \int_{-\alpha_0}^{\alpha_0}
V([x_0+t_0v,\ldots,x_d+t_dv])\,dt_0\ldots dt_d\\
&\qquad- \int_{-\alpha_d}^{\alpha_d} \int_{-\alpha_{d-1}}^{\alpha_{d-1}}\ldots
\int_{-\alpha_0}^{\alpha_0}
V([x_0+t_0v,\ldots,x_d+t_dv])^p\,dt_0\ldots dt_d\\
&\geq \delta^2 \, \frac{p2^{d-p}}{d^p}\,\alpha_0\ldots\alpha_{d-1}\alpha_d^{p-1}\,
V_{d-1}(\pi_v[x_0,\ldots,x_{d-1}])^p.
\end{align*}
\end{itemize}
\end{lemma}

\begin{proof}
We prove only (ii); obtaining (i) by the same method is straightforward. Due to
condition (\ref{intcond}) and symmetry, and by using the  notation
\begin{align*}
\omega(t_0, t_1, \dots , t_d)=& V([x_0+t_0v,\ldots,x_{d-1}+t_{d-1}v,x_d+t_dv])^p\\
&+ V([x_0-t_0v,\ldots,x_{d-1}-t_{d-1}v,x_d+t_dv])^p,
\end{align*}
the following
holds:
\begin{align*}
&2\int_{-\alpha_d+\delta}^{\alpha_d+\delta}
\int_{-\alpha_{d-1}}^{\alpha_{d-1}}\ldots \int_{-\alpha_0}^{\alpha_0}
V([x_0+t_0v,\ldots,x_d+t_dv])^p\,dt_0\ldots dt_d \\
&\qquad -2\int_{-\alpha_d}^{\alpha_d} \int_{-\alpha_{d-1}}^{\alpha_{d-1}}\ldots
\int_{-\alpha_0}^{\alpha_0} V([x_0+t_0v,\ldots,x_d+t_dv])^p\,dt_0\ldots dt_d\\
&=\int_{\alpha_d-\delta}^{\alpha_d}\ldots \int_{-\alpha_0}^{\alpha_0}
\omega(t_0, \dots ,t_{d-1}, t_d+\delta)-
 \omega(t_0, \dots,t_{d-1} ,t_d )\,dt_0\ldots dt_d.
\end{align*}
For fixed $t_i\in[-\alpha_i,\alpha_i]$, $i=0,\ldots,d-1$
and $t_d\in [\alpha_d-\delta,\alpha_d]$,
let $s\in [-\alpha_d+\delta,\alpha_d-\delta]$ satisfy that
$x_0+t_0v$,\ldots,$x_{d-1}+t_{d-1}v$ and $x_d+sv$
are contained in a hyperplane. It follows that
\begin{align*}
\omega(t_0, \dots ,t_{d-1}, t_d+\delta)-
 \omega(t_0, \dots,t_{d-1} ,t_d )&=\\
\frac{V_{d-1}(\pi_v[x_0,\ldots,x_{d-1}])^p}{d^p}\times&\\
[(t_d+\delta+s)^p+(t_d+\delta-s)^p-(t_d+s)^p
-(t_d-s)^p].&
\end{align*}
We claim that
\begin{equation}
\label{tsdelta}
(t_d+\delta+s)^p+(t_d+\delta-s)^p-(t_d+s)^p
-(t_d-s)^p\geq p\delta\alpha_d^{p-1}/2^{p-1}.
\end{equation}
We may assume that $s\geq 0$, and hence $s\in[0,t_d]$. Let $\psi(s)$ be the
left hand side of (\ref{tsdelta}) as a function of $s$, then
$$
\psi'(s)=p(t_d+\delta+s)^{p-1}-p(t_d+\delta-s)^{p-1}-
[p(t_d+s)^{p-1}-p(t_d-s)^{p-1}].
$$
Since $p\tau^{p-1}$ is convex, if $p\geq 2$, and
 concave, if $1\leq p< 2$ for $\tau>0$, we deduce that $\psi'$
is non-negative, hence $\psi$
is increasing, if $p\geq 2$, and $\psi'$
is non-positive, hence $\psi$ is
 decreasing, if $1\leq p< 2$. In particular, we may assume
$s=0$, if $p\geq 2$, and
 $s=t_d$, if $1\leq p< 2$ in (\ref{tsdelta}).
Therefore the estimates $t_d\geq \alpha_d/2$
and $(\tau+\delta)^p-\tau^p>p\delta \tau^{p-1}$
for $\tau=t_d$ or $\tau=2t_d$ yield (\ref{tsdelta}).
In turn we conclude Lemma~\ref{steinerstab}.
\end{proof}

\section{Stability of the minimum inequalities}

We are going to use Vinogradov's $\gg$ notation in the following sense: $f \gg
g$ or $g\ll f$ for non-negative functions $f$ and $g$
iff there exists a constant $c>0$ depending only on $d$, for which $f \geq c g$ holds.
In addition, we write $h=O(f)$ if $|h|\ll f$.

We will say that a convex body $K \subset \R^d$ is in {\em John position}, if
its unique inscribed ellipsoid of maximal volume is $B^d$.
 We are going to use the following simple consequence of Fritz John's
ellipsoid theorem (see \cite{Joh48} and \cite{Bal97}).

\begin{prop}\label{john}
Assume that the $o$-symmetric convex body $K \subset \R^d$ is in John
position. Then for any point $p \in S^{d-1}$, there is a contact point $q$
between $K$ and $B^d$, for which $\la p,q \ra \geq 1/\sqrt{d} $.
\end{prop}

The statement is equivalent to the well-known fact that any point in $K$ has
norm at most $\sqrt{d}$.

We will use the following notations. Let $K$ be a convex body in $R^d$. Let $H$
be a hyperplane of $R^d$ with normal $v$. Let $\ell$ be the line of direction
$v$, and for any $x \in H$, denote by $\sigma(x)$ the secant $K \cap (x +
\ell)$, and by $M(x)$ the midpoint of $\sigma(x)$. Moreover, let $m(x)$ be the
signed distance of $x$ and $M(x)$, that is,  $m(x) = \la M(x) - x, v \ra$.

Now, for Theorem~\ref{ellipsoidstab}. First, we deal with the case when $K$ is
$o$-symmetric and its Banach-Mazur distance from $B^d$ is sufficiently
small. This is the core of the proof.

\begin{lemma}
\label{ellstablocal} For any $d\geq 2$, there exists  $\eps_0,\hat{\gamma}>0$,
such that if $K\subset \R^d$ is an $o$-symmetric convex body in John position,
and the maximal norm of the points of $K$ is $1+\eps$ with $\eps\leq\eps_0$,
then for any $p \geq 1$,
\begin{align*}
\E^p_o(K) - \E^p_o(B^d) &\geq \hat{\gamma}^p\eps^{(d+3)/2} \textrm{, and}\\
\E^p_{d+1}(K) - \E^p_{d+1}(B^d) &\geq \hat{\gamma}^p \eps^{(d+3)/2}.
\end{align*}
\end{lemma}

\begin{proof} Let $r$ be a point of $K$ of maximal
norm. By Proposition~\ref{john}, there is a contact point $q\in\partial K\cap
S^{d-1}$ with $\langle-r,q\rangle\geq \|r\|/\sqrt{d}$. Let $\ell$ be the line
passing through
$r,q$ with direction vector $v=(r-q)/\|r-q\|$, let $H = v^\perp$, and choose
a coordinate system such that the $d$th coordinate axis is parallel to $\ell$.
 Taking $x_d=\pi_v r=\pi_vq$, a
simple calculation shows that
\begin{equation}\label{xdsize}
\|x_d\|< \frac{1}{\sqrt{2}}-\frac{1}{4 \sqrt{d}}\;.
\end{equation}
For any $x \in H\cap B^d$, let $\sigma(x)= K \cap (x+\ell)$ with midpoint
$M(x)$, and define $m(x) = \la (M(x)-x),v\ra$. Since $B^d \subset K \subset (1+
\eps) B^d$, if $\|x\|\leq 0.9$, then $m(x)$ can be estimated as
\begin{equation}\label{midest}
|\,m(x)| \leq \frac{\sqrt{(1+ \eps)^2 - \|x\|^2} - \sqrt{1-\|x\|^2}}{2}
= \frac{\eps(1+O(\eps))}{2 \sqrt{1-\|x\|^2}}.
\end{equation}
Note that for $x=x_d$, equality holds
in \eqref{midest}.

The estimating function
is illustrated on Figure~\ref{fv1}.

\begin{figure}[h]
\epsfxsize =7 cm \centerline{\epsffile{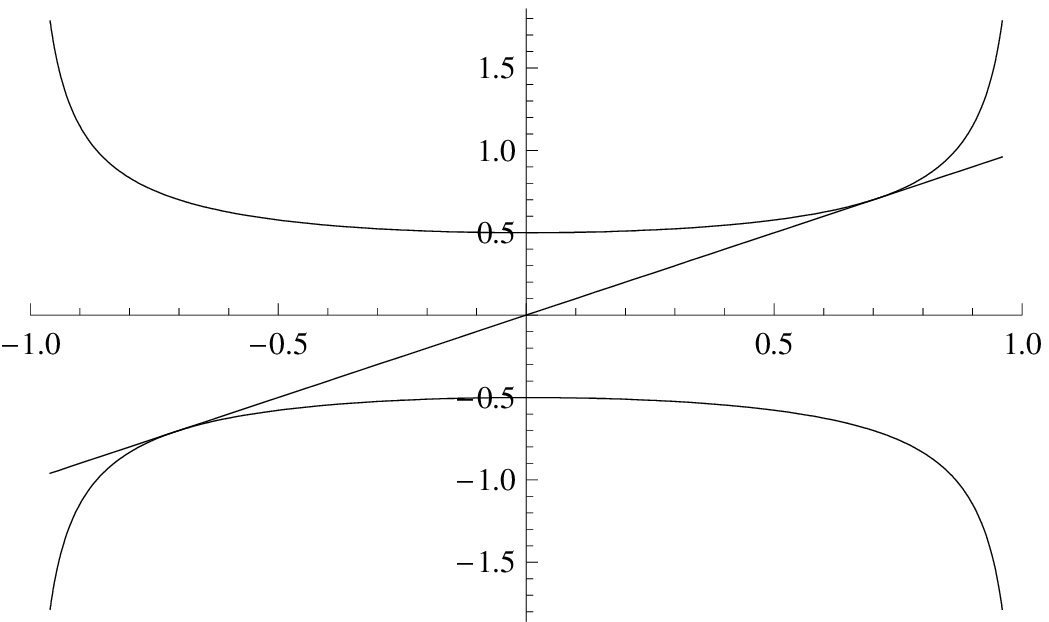}}
  \caption{}
  \label{fv1}
\end{figure}


The tangent from $o$ to the graph of $f(z)=1/\sqrt{1-z^2}$ has its contact
point at $z=1/\sqrt{2}$. Due to the convexity of $f(z)$, estimates
(\ref{xdsize}) and (\ref{midest}) imply that if we choose the points $x_1,
\dots, x_{d-1}$ of norm about $1/\sqrt{2}$ with $x_d \in [o,x_1, \dots,
x_{d-1}]$, then $M(x_d)$ is separated from $[o, M(x_1), \dots, M(x_d)]$ by $c
\eps$, where $c$ is a constant depending on $d$ only. This then yields a
positive error $\E^p_o(K)$ in comparison with $\E^p_o(B^d)$. This idea is
transformed to a quantitative proof as follows.

First, we estimate the decay of $m(x)$ around $x_d$. By convexity, $[B^d,
r\,]\subset K$. Let $\hat{r} = S^{d-1} \cap [o\,, r\,]$, and $\tilde{r}= S^{d-1}
\cap ([q \, r\,] \setminus q)$.  Estimate (\ref{xdsize}) yields that $\|
\hat{r} -\tilde{r}\| \leq \eps$. For $s \in S^{d-1}$, denote by $T(s)$ be the
tangent hyperplane to $S^{d-1}$ at $s$. It is easily obtained that the
intersection $[B^d, r\,] \cap T(\hat{r})$ is a $(d-1)$-dimensional ball of
radius $\sqrt{\eps/(2+\eps)}$, and thus, $A=[B^d, r\,] \cap T(\tilde{r})$
contains a ball of radius $\sqrt{\eps/2.5}$ centred at $\tilde{r}$. Then, again
by (\ref{xdsize}), $\pi_v(A)$ contains a ball $D$ of radius $\sqrt{\eps}/4$
centred at $x_d$. Since $T(q)$ is a tangent hyperplane of $K$, $m(x)$ can be
estimated over $D$ linearly:
\begin{equation}
\label{mest} m\left(x_d + t \frac{\sqrt{\eps}}{4} u\right)\geq (1-t) m(x_d), \
\forall \; u \in S^{d-1}, \forall \; t\in[0,1].
\end{equation}

Next, we are going to estimate $\E^p_o(K)-\E_o^p(B^d)$.  Let $x_0=o$, and
choose $x_1, \dots, x_{d-1}$ as follows. Take $\tilde{y} = x_d / \| x_d \|$. If
$d=2$, then let $x_1 = \tilde{y}/\sqrt{2}$. If $d \geq 3$, then take $y =
(1/\sqrt{2}-1/(100d)) \tilde{w}$, and let $x_1, \dots, x_{d-1}$ be of norm
$1/\sqrt{2}-1/(500d)$, the vertices of a regular $(d-2)$-simplex in $(y +
y^{\perp}) \cap H$ with centroid $y$. Note that the distance between any two of
these
 is $>1/100\sqrt{d}$. Let $\varrho=1/(1000 d)$ and define $X_i = x_i + \varrho B^{d-1}
\subset H$ for every $i=1, \dots, d-1$. Then $V_{d-1}(X_i) \gg 1$.

Note that by (\ref{xdsize}), there exists a neighbourhood $U$ of $x_d$ of
radius $\gg \sqrt{\eps}$ in $H$ such that for any $x_i' \in X_i,\ i=1, \dots,
d-1$, we have $U \subset [o, x_1', \dots, x_{d-1}']$. For such a collection of
$(x_i')$, and for any $x_d' \in U$, define
\[
D((x_i'))= D(x_1', \dots, x_{d-1}',x_d') = [o, M(x_1'), \dots, M(x_{d-1}')]
\cap \sigma (x_d'),
\]
and let $d((x_i')) = \la D((x_i')), v \ra$. Note that for any $x_i' \in X_i$,
$1 \leq i \leq d-1$,
\[
\frac{1}{\sqrt{2}} - \frac{3}{500 d}\leq \| x_i' \| \leq \frac{1}{\sqrt{2}}
-\frac{1}{500d}.
\]
Thus, (\ref{midest}) yields that there exists a neighbourhood $V\subset U$ of
$x_d$ in $H$, still of area $\gg \eps^{(d-1)/2}$, such that for any $x_i' \in
X_i$, $i=1,\dots,d-1$ and any $x_d' \in V$, for sufficiently small $\eps$ we
have
\begin{equation}\label{dxest}
d(x_1', \dots, x_d') \leq \frac {\| x_d \| \eps}{1 - 1/ (100d)}.
\end{equation}
Since
\[
\frac{\eps}{2\sqrt{1-\|x_d\|^2}} - \frac {\| x_d \| \eps}{1 - 1/ (100d)}
\]
as a function of $\|x_d\|$ is decreasing for $\|x_d\|<1/\sqrt{2}$, estimates
(\ref{xdsize}), (\ref{midest}) and (\ref{dxest}) yield that for  $x_i' \in X_i$
and $x_d' \in V$,
\begin{align*}
m(x_d) - d(x_1', \dots, x_d')&\geq \eps \left(\frac{1}{\sqrt{2} +  1/(2
\sqrt{d})} - \frac{1/\sqrt{2} - 1/(4\sqrt{d})}{1-1/(100d)}\right) \\
&\geq \frac{\eps}{20 d} \; .
\end{align*}
Let $R = \sqrt{ 2\eps} / (100d)$, and take $X_d = V \cap (x_d + R B^{d-1})
\subset H$. Then $V_{d-1}(X_d)\gg \eps^{(d-1)/2}$. Moreover, since $m(x_d) <
\eps/\sqrt{2}$, the above estimate and (\ref{mest}) yield that for $x_i' \in
X_i$, $i=1, \dots, d$,
\begin{equation}\label{mfinalest}
m(x_d') - d((x_1', \dots, x_d')) \geq \frac{\eps}{100d}.
\end{equation}
Let now $K'$ be the Steiner symmetral of $K$ with respect to $H$. By
Theorem~\ref{CCG}, it is sufficient to prove that $\E^p_o(K)-\E^p_o(K') \geq
\hat\gamma^p \eps^{(d+3)/2}$. We calculate the average volume of random
simplices by integrating along the $d$-tuples of chords of $K$ parallel to $v$.
For $x \in H$, let $\sigma_K(x) = \sigma (x) = K \cap (x+ \ell)$, and
$\sigma_{K'}(x) = K' \cap (x + \ell)$. For $x_1', \dots, x_d' \in H \cap K$,
define
\begin{multline*}
\omega(x_1', \dots, x_d') = \int_{\sigma_K(x_1')} \dots
\int_{\sigma_K(x_d')}V[o,y_1, \dots, y_d]\,dy_d \dots dy_1\\
- \int_{\sigma_{K'}(x_1')} \dots \int_{\sigma_{K'}(x_d')}V[o,y_1, \dots,
y_d]\,dy_d \dots dy_1
\end{multline*}
Lemma~\ref{steiner} yields that for any $(x'_i)_1^d \subset H \cap K$, we have
$\omega(x_1', \dots, x_d') \geq 0$. Moreover, by the construction of
$(X_i)_1^{d-1}$,  for any $x_i' \in X_i$, we have
$V_{d-1}([o, x_1', \dots, x_{d-1}'])\gg 1$. Thus, by (\ref{mfinalest}),
Lemma~\ref{steiner}, and part (i) of Lemma~\ref{steinerstab},
\[
\E_o^p(K) - \E_o^p(K') \geq \int_{X_1} \dots \int_{X_d} \omega(x_1', \dots,
x_d') \, dx_d' \dots dx_1' \geq \gamma_1^p \eps^{(d+3)/2}
\]
for some $\gamma_1>0$ depending only on $d$.

Next, we estimate $\E^p_{d+1} (K)  - \E^p_{d+1}(K')$. We start as before. There
are  two cases to be considered depending on $\| x_d \|$. First, assume that
$\| x_d \| \geq 1/100$ (we need only $\|x_d\| \gg 1$). Then construct
$(X_i)_1^{d}$ as before. Choose $R>0$ small enough such that the following
hold:
\begin{itemize}
\item[i)] For any $x'_0$ with $\|x_0'\|\leq R$, and any $x'_i \in X_i$, $i=1,
\dots, d$, we have $x'_d \in [x'_0, \dots, x'_{d-1}]$

\item[ii)] For any $x'_0$ with $\|x_0'\|\leq R$ and $m(x'_0) \leq 0$, and any $x'_i \in X_i$, $i=1,
\dots, d$, \begin{equation}\label{mM} m(x'_d) - \la [M(x_0'), \dots,
M(x'_{d-1})]\cap \sigma(x'_d),v\ra \gg \eps.\end{equation}
\end{itemize}
Let $X_0 = \{ x \in H: |\,x| < R, \ m(x)\leq 0\}$. By the symmetry of $K$, the
measure of $X_0$ is at least half as large as that of $R B^{d-1}$, thus,
$V_{d-1}(X_0) \gg 1$. Then, part (ii) of Lemma~\ref{steinerstab} applies as
before, yielding
\[
\E^p_{d+1} (K) - \E^p_{d+1} (K')\geq \gamma_2^p \eps^{(d+3)/2}
\]
for some $\gamma_2>0$ depending only on $d$.

In the second case, $x_d$ is close to the origin: $\| x_d \| < 1/100$. Let $A$
be the annulus $\{x \in H: 1/2<\|x\|<3/4\}$. For this instance, define the
function $d'$ on $A^d$ by
\[
d'(x'_0, \dots, x'_{d-1}) = \la([M(x'_0), \dots, M(x'_{d-1})]\cap \sigma(o))
,v\ra.
\]
Note that by symmetry, $d'(-x'_0, \dots, -x'_{d-1})=-d'(x'_0, \dots,
x'_{d-1})$. Let $C=1/100$, and consider only those $(x'_i)_0^{d-1} \subset A$,
for which $|\la u,v \ra| \geq C$, where $u$ is the normal vector of $[M(x'_0),
\dots, M(x'_{d-1}))]$. Then the (product) measure of these point sets is $\gg
1$; moreover, at least half of them satisfies $d'(x'_0, \dots, x'_{d-1})\leq0$.
These also satisfy \eqref{mM}. Thus, integrating over these sets, the argument
works as before.
\end{proof}

\noindent {\bf Remark}. In the planar case, one can obtain the following
quantitative result: If $K$ satisfies the conditions of
Lemma~\ref{steinerstab}, then for small $\eps>0$,
\[
\E^1_o(K) - \E^1_o(B^2) > \frac{\eps^{5/2}}{400} \; .
\]
From this, it also follows that if $K$ is a centrally symmetric convex disc,
and $\E^1_o(K) \leq (1+ \delta) \E^1_o(B^2)$, then there exists an ellipse $E$,
for which $E \subset K \subset (1 + 20 \delta^{2/5})E$.
\bigskip

To obtain the estimate for not necessarily symmetric bodies, we cite the
following result of the second author, see Theorem~1.4 of \cite{Bor10}.

\begin{lemma}\label{rounding}
For any convex body $K \subset \R^d$ with $\delta_{\rm BM}(K,B^d)> 1+\eps$ for
some  $\eps>0$, there exists an $o$-symmetric convex body $C$ with axial
rotational symmetry and a constant $\gamma>0$ depending only on $d$, such that
$\delta_{\rm BM}(C,B^d)>\gamma\eps^2$, and $C$ results from $K$ as a limit of
subsequent Steiner symmetrisations and affine transformations.
\end{lemma}

Now, we are ready to prove the general result.

\begin{proof}[Proof of Theorem~\ref{ellipsoidstab}]
Let $ \delta_{\rm BM}(K,B^d) =1+\delta $.
By Lemma~\ref{rounding}, we may assume that $K$ is an
$o$-symmetric convex body in John position, provided we prove
\begin{align}
\label{goal*}
\E^p_o(K)&\geq(1+\gamma^p\delta^{\frac{d+3}2}) \E^p_o(B^d) \textrm{, and}\\
\label{goald}
\E^p_{d+1}(K)&\geq(1+\gamma^p\delta^{\frac{d+3}2}) \E^p_{d+1}(B^d)
\end{align}
for $\gamma>0$ depending only on $d$.
Let the maximal norm of points of $K$
be $1+\eps$. Since the volume of $K\setminus B^d$ is $\gg \eps^{(d+1)/2}$,
it follows that
\begin{equation}
\label{deltaeps}
\gamma_0\varepsilon^{\frac{d+1}2}\leq  \delta  \leq \varepsilon
\end{equation}
for $\gamma_0>0$ depending only on $d$.

Let $\eps_0$ and $\hat{\gamma}$ come from Lemma~\ref{ellstablocal}.
If $\delta \leq \delta_0=\gamma_0\varepsilon_0^{\frac{d+1}2}$ then
$\varepsilon \leq \varepsilon_0$ by \eqref{deltaeps},
and hence we have \eqref{goal*} and \eqref{goald}
with $\gamma=\hat{\gamma}$ by Lemma~\ref{ellstablocal}
and \eqref{deltaeps}.

Therefore we may assume that $\delta >  \delta_0$. Choose a sequence of Steiner
symmetrals $K_0,K_1, K_2, \dots$ starting with $K=K_0$ that converge to $B^d$,
and hence there exists $K_n$ such that $\delta_{BM}(K_{n+1})\leq \delta_0<
\delta_{BM}(K_n)$. Let $L_t : t \in[-1,1]$ be the linear shadow system with
$L_1=K_n$ and $L_0=K_{n+1}$  corresponding to the Steiner symmetrisation of
$K_n$ (see Section~\ref{secshadow}), thus  there exists $t\in[0,1)$ such that
$\delta_{BM}(L_t)= \delta_0$. It follows that
 $\E^p_*(L_t)\leq \E^p_*(K_n)\leq \E^p_*(K)$ and
$\E^p_{d+1}(L_t)\leq \E^p_{d+1}(K_n)\leq \E^p_{d+1}(K)$,
thus we conclude \eqref{goal*} and \eqref{goald}
by the  previous case and $\delta<\sqrt{d}$.
\end{proof}

We made no attempt to find the best possible constants. However, the estimate
$\eps^{(d+3)/2}$ for centrally symmetric $K$ is close to the truth: if $K=[r,
-r, B^d]$, where $r$ is of norm $1+\eps$, then
\[
\frac{\E^p_o(K)}{\E^p_o(B^d)} - 1 \ll \eps^{(d+1)/2.}
\]

\section{Stability of the maximum inequalities in the plane}
\label{secplane}

Since here we work only on the plane, a {\em convex disc} means a planar convex
body, and $A(K)=V_2(K)$ is the area of $K$.  For a polygon $\Pi$ with at least
four vertices $q_1,\ldots,q_k$ in this order, a {\em basic linear shadow
system} at $q_1$, basic system for short, is defined as follows. Let $q'_1$ and
$q''_1$ be points different from $q_1$ such that $q_1\in [q'_1,q''_1]$,
$q'_1-q''_1$ is parallel to $q_2-q_k$, and $q_2,\ldots,q_k$ lie on the boundary
of $\Pi'=[q'_1,q_2,\ldots,q_k]$ and $\Pi''=[q''_1,q_2,\ldots,q_k]$. The
corresponding basic system is the unique linear shadow system $\Pi_t$,
$t\in[-\beta,\alpha]$,  such that $\alpha,\beta>0$, $\alpha+\beta=1$,
$\Pi_{-\beta}=\Pi''$, $\Pi_{0}=\Pi$, and $\Pi_{\alpha}=\Pi'$. In this case, the
generating vector is parallel to $q_2-q_k$, and the speed of any point in
$[q_2,\ldots,q_k]$ is zero. It follows from Theorem~\ref{CCG} that for any
$n\geq 3$ and $p\geq 1$,
\begin{equation}
\label{basicmax}
\E^p(\Pi)<\max\{\E^p(\Pi'),\E^p(\Pi'')\},
\end{equation}
where $\E^p(\Pi)$ stands either for $\E^p_o(\Pi)$ or $\E^p_n(\Pi)$.
 More precisely, the following holds:
\begin{equation}
\label{basicdef}
\mbox{ $\E^p(\Pi_{-t})$ on $[0,\beta]$, or $\E^p(\Pi_t)$ on
$[0,\alpha]$, is strictly increasing.}
\end{equation}

For a convex disc $K$, let $T_K$ be  a triangle of maximal area contained in
$K$. It follows that the triangle, the midpoints of whose sides are the
vertices of $T_K$, contains $K$. In particular, $A(K)<4 A(T_K)$.

First, we reduce the case to polygons with at most 6 vertices.

\begin{prop}
\label{trianglehexagon}
For a convex disc $K$, let $\widetilde{T}$ be
the triangle, the midpoints of whose sides are the vertices of $T_K$.
 For $n\geq 3$ and $p\geq 1$,  there exist polygons
$\Pi_1$ and $\Pi_2$ with $A(\Pi_1)=A(\Pi_2)=A(K)$ such that $T_K\subset
\Pi_1,\Pi_2$, all vertices of $\Pi_1,\Pi_2$ are on $\partial \widetilde{T}$,
and
$$
\E^p_n(K)\leq\E^p_n(\Pi_1) \mbox{ \ and \ }\E^p_*(K)\leq\E^p_*(\Pi_2).
$$
\end{prop}
\begin{proof} We may assume that $K$ is not a triangle, and by continuity, that $K$ is
a polygon. However, for $k\geq 4$, suitable basic systems and \eqref{basicmax}
yield that among polygons $P$ of at most $k$ vertices with fixed area such that
$T_K\subset P\subset\widetilde{T}$, any polygon maximising either $\E^p_n(P)$
or $\E^p_*(P)$ has all of its vertices in $\partial \widetilde{T}$.
\end{proof}
The core lemma comes.

\begin{lemma}
\label{localmax}
There exist positive absolute constants $\varepsilon_0,\hat{c}$
such that if $p\geq 1$, and  $A(K)=(1+\varepsilon)A(T_K)$ for a convex disc $K$
and $\varepsilon\in(0,\varepsilon_0]$, then
$$
\E^p_3(K)\leq(1-\hat{c}^p\varepsilon^2)\E^p_3(T^2) \mbox{ \ and \
}\E^p_*(K)\leq(1-\hat{c}^p\varepsilon^2)\E^p_*(T^2).
$$
\end{lemma}
\begin{proof} We first consider $\E^p_*(K)$. Let $T_K=[p_1,p_2,p_3]$, and let
$q_1,q_2,q_3$ be the such that $p_i$ is the midpoint of $[q_j,q_k]$,
$\{i,j,k\}=\{1,2,3\}$. We may assume that each side of $T_K$ is of length one,
and $\gamma(T_K)=o$.

Let $\Pi$ be the polygon provided by Claim~\ref{trianglehexagon}, and let $x$
be the farthest vertex of $\Pi$ from $T_K$.
 We may assume that
$x\in[p_1,q_2]$. It follows that $A([x,p_1,p_3])$ is between $\varepsilon
A(T_K)/6$ and $\varepsilon A(T_K)$, and hence
\begin{equation}
\label{xp1} \varepsilon/6\leq \|x-p_1\|\leq \varepsilon.
\end{equation}

Let us number the vertices of $\Pi$ in such a way  that $x=x_3$, its
neighbouring vertices are $x_2\in[p_1,q_3]$ and $x_4\in[p_3,q_2]$, and the
other neighbours of $x_2$ and $x_4$ are $x_1\in[p_2,q_3]$ and $x_5$,
respectively; see Figure~\ref{triangle}. Here possibly $x_5=x_1$, and either
$x_5\in[p_3,q_1]$, or $x_5\in[p_2,q_1]$. The definition of $x=x_3$ yields that
for any $i=1,2,4,5$ there exists $j\in\{1,2,3\}$ such that
\begin{equation}
\label{xipj} \|x_i-p_j\|\leq   \|x-p_1\|.
\end{equation}

To deform $\Pi$, let $l$ be the line parallel to $x_2-x_4$ passing through $x$,
and let $x'$ and $x''$ be the intersections of $l$ with ${\rm aff}\{x_2,x_1\}$
and ${\rm aff}\{x_4,x_5\}$, respectively. We consider the basic system $\Pi_t$,
$t\in[-\beta,\alpha]$, $\alpha,\beta>0$, $\alpha+\beta=1$, where $\Pi_0=\Pi$,
$x'$ is a vertex of $\Pi_{\alpha}$, and $x''$ is a vertex of $\Pi_{-\beta}$. We
write $\varphi(z)$ to denote the speed of a $z\in \Pi$, and observe that the
generating vector is $v=\frac{x_2-x_4}{\|x_2-x_4\|}$.

\begin{figure}[h]
\epsfxsize =9 cm \centerline{\epsffile{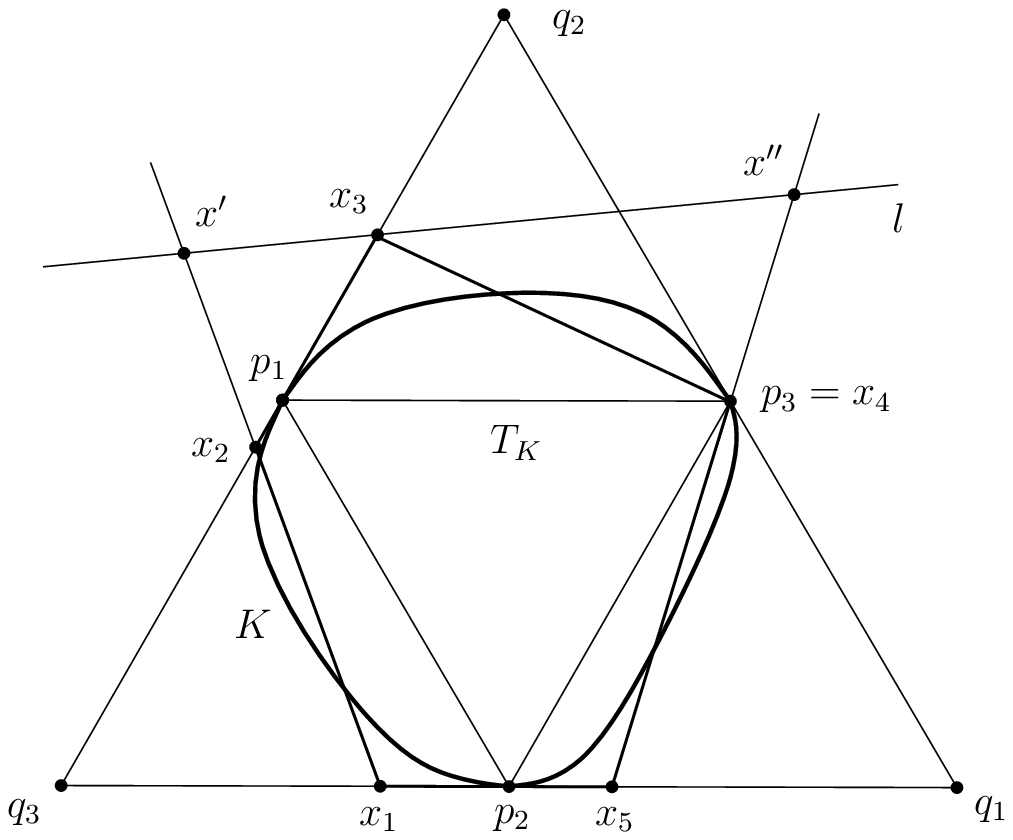}}
  \caption{}
  \label{triangle}
\end{figure}


It follows by \eqref{shadowcentroid} that
$\gamma(\Pi)=\beta\,\gamma(\Pi_{\alpha})+\alpha\,\gamma(\Pi_{-\beta})$. Thus
for any $z_1,z_2\in\Pi$, Theorem~\ref{shadow} yields
\begin{multline}\label{generalarea}
A([\gamma(\Pi),z_1,z_2])^p\leq \beta
A([\gamma(\Pi_{\alpha}),z_1+\alpha\varphi(z_1)v,z_2+\alpha\varphi(z_2)v])^p\\
+\alpha
A([\gamma(\Pi_{-\beta}),z_1-\beta\varphi(z_1)v,z_2-\beta\varphi(z_2)v])^p.
\end{multline}

In order to obtain a stability statement, we improve on \eqref{generalarea}. As
a first step, we localise $\gamma(\Pi_t)$. The centroid $\gamma(\Pi)$ has the
property that $-1/3 \,(\Pi-\gamma(\Pi))\subset \Pi-\gamma(\Pi)$. It
follows by \eqref{xp1} and \eqref{xipj} that
\begin{equation}
\label{bPi} \gamma(\Pi)\in 2\varepsilon T_K.
\end{equation}
We note that by \eqref{xipj}, $|\varphi(z)|\leq 1.1$  for
$z\in \Pi$, and $\varphi(z)=0$ if $z$ is separated from $x=x_3$ by the diagonal
$[x_2,x_4]$. Thus \eqref{shadowcentroid} yields
\begin{equation}
\label{bPit} \gamma(\Pi_t)=\gamma(\Pi)+t\omega v,\mbox{ \ for
$\omega\in(0,2\varepsilon)$ independent of $t$.}
\end{equation}

As $[p_1,p_3]$ is close to $l$ (any $z\in[p_1,p_2]$ is of distance at most
$3\varepsilon$ from~$l$) and  $[p_1,p_2]$ is close to $[x_2,x_1]$, we may
choose $\varepsilon_0$ small enough to ensure $\varepsilon/12\leq \|x-x'\|\leq
2\varepsilon$. In addition, $[x_4,x_5]$ is either contained in $[q_2,q_1]$, or
it is close to $[p_3,p_2]$, therefore $2/3\leq \|x-x''\|\leq 3/2$. We deduce
\begin{equation}
\label{alpha} \varepsilon/24\leq \alpha\leq 4\varepsilon.
\end{equation}

We may assume that $\R=v^\bot$, oriented in a way such that $\pi_vp_3>0$. We
observe that
$$
\frac1{2\sqrt{3}}-\varepsilon<\pi_vp_3\leq\frac1{2\sqrt{3}}\leq\pi_vp_1
<\pi_vx_3-\frac{\varepsilon}{12}.
$$
For $y\in\pi_v{\rm int}\Pi$ and $t\in[-\beta,\alpha]$, we write $\sigma_t(y)$
to denote the chord of $\Pi_t$ parallel to $v$ and projecting into $y$, and
$m_t(y)$ to denote the midpoint of $\sigma_t(y)$. In particular,
$\sigma_t(y)=\sigma_0(y)$ if $y\leq\pi_vx_2$.
 If $\varepsilon_0$ is
small enough then for any $s\in(0,\frac18)$,
$$
|\langle v,m(\pi_vp_3-s)\rangle|\leq 2s \mbox{ and
$V_1(\sigma(y))>\frac{\sqrt{3}}2$.}
$$
We consider the  intervals
$$
\mbox{$I_1=[\pi_vp_3-\frac1{16},\pi_vp_3-\frac1{32}]$ and
$I_2=[\frac1{16}\pi_vp_1+\frac{15}{16}\pi_vx_3,
\frac1{32}\pi_vp_1+\frac{31}{32}\pi_vx_3]$},
$$
and hence \eqref{xp1} yields
\begin{equation}
\label{Ilength} \mbox{$V_1(I_1)=\frac1{32}$ and $V_1(I_2)\geq
\frac{\varepsilon}{12\cdot 64}$}.
\end{equation}
In addition, $\sigma_t(y)=\sigma_0(y)$ if $y\in I_1$ and $t\in[-\beta,\alpha]$.
To ensure the condition (\ref{intcond}) in Lemma~\ref{steinerstab}, for $y\in
I_1$, we restrict our attention to
$$
\sigma^*_t(y)=\mbox{$\frac18$}(\sigma_t(y)-m_t(y))+m_t(y).
$$
Our  main claim is that there exists an absolute constant $c_1>0$, such that
for any $y_1\in I_1$ and $y_2\in I_2$, the integral
\[
f(t)=\int_{\sigma^*_t(y_1)}\int_{\sigma_t(y_2)}
A([\gamma(\Pi_t),z_1,z_2])^p\,dz_1dz_2
\]
satisfies
\begin{equation}
\label{segmentstab} \alpha f(-\beta)+\beta f(\alpha)\geq f(0)+c_1^p\varepsilon.
\end{equation}
It follows by \eqref{bPi} and \eqref{bPit} that if $\varepsilon_0$ is small
enough, then there exists a $\tau\in(\frac14,\frac34)$, such that
$\gamma(\Pi_{-\tau})$, $m_{-\tau}(y_1)$ and $m_{-\tau}(y_2)$ are collinear.
Writing $\omega_t$ to denote the intersection point of ${\rm
aff}\{\gamma(\Pi_t),m_t(y_1)\}$ and ${\rm aff}\sigma_t(y_2)$, the function $\la
v,\omega_t- m_t(y_2)\ra$ of $t$ is linear, zero at $-\tau$, and satisfies
\[
\mbox{$\langle v,\omega_{\alpha}- m_{\alpha}(y_2)\rangle\geq\frac18$ and
$\langle v,\omega_{-\beta}- m_{-\beta}(y_2)\rangle\leq -\frac18$}.
\]
We deduce by Lemma~\ref{steiner} and \eqref{translateo} that $f(t)$ is convex,
and has its minimum at $-\tau$. Thus Lemma~\ref{steinerstab} yields
$$
f(\alpha),f(-\beta)\geq f(-\tau)+c_2^p \mbox{ \ for an absolute constant
$c_2>0$.}
$$
It follows by $\beta=1-\alpha$ and \eqref{alpha} that
\begin{align*}
\alpha f(-\beta)+\beta f(\alpha)-f(0)  &\geq \alpha f(-\beta)+\beta f(\alpha)-
\mbox{$\frac{\alpha}{\alpha+\tau}\,f(-\tau)-
\frac{\tau}{\alpha+\tau}\,f(\alpha) $}\\
&=\alpha f(-\beta)+
\mbox{$\alpha\cdot\frac{1-\alpha-\tau}{\alpha+\tau}\,f(\alpha)
-\frac{\alpha}{\alpha+\tau}\,f(-\tau)$}\\
&\geq \mbox{$\frac{\alpha}{\alpha+\tau}\cdot c_2^p\geq
\frac{c_2^p}{24}\cdot\varepsilon$}.
\end{align*}
Therefore we have verified \eqref{segmentstab}. In turn combining this with
\eqref{generalarea} and \eqref{Ilength} proves for a suitable absolute constant
$c_3>0$, that
$$
\E^p_*(\Pi)+c_3^p\varepsilon^2\leq
\beta\E^p_*(\Pi_\alpha)+\alpha\E^p_*(\Pi_{-\beta})\leq
\max\{\E^p_*(\Pi_\alpha),\E^p_*(\Pi_{-\beta})\}.
$$
Applying subsequent basic systems to the one  of $\Pi_\alpha$ and
$\Pi_{-\beta}$ with larger $\E^p_*(\cdot)$, we conclude
$$
\E^p_*(K)+c_3^p\varepsilon^2\leq\E^p_*(\Pi)+c_3^p\varepsilon^2\leq \E^p_*(T^2).
$$

Turning to $\E^p_3(K)$, the major difference of the argument is that we need a
third interval for the third vertex of the triangle. Writing $I_1=[a,b]$, we
define $\widetilde{I}_2=I_2$, and
$$
\mbox{$\widetilde{I}_0=a+\frac1{10}(I_1-a)$ and
$\widetilde{I}_1=b+\frac1{10}(I_1-b)$}.
$$
In addition, we shorten $\sigma^*_t(y)$ for $y\in I_1$ to
$$
\tilde{\sigma}_t(y)=\mbox{$\frac1{80}$}(\sigma_t(y)-m_t(y))+m_t(y).
$$
We change our  main claim \eqref{segmentstab} to the following. There exists an
absolute constant $c_4>0$, such that for any $y_0\in \widetilde{I}_0$, $y_1\in
\widetilde{I}_1$ and $y_2\in \widetilde{I}_2$, the integral
$$
\tilde{f}(t)=
\int_{\tilde{\sigma}_t(y_0)}\int_{\tilde{\sigma}_t(y_1)}\int_{\sigma_t(y_2)}
A([z_0,z_1,z_2])^p\,dz_0dz_1dz_2
$$
satisfies
\begin{equation}
\label{segmentstabtri} \alpha \tilde{f}(-\beta)+\beta \tilde{f}(\alpha)\geq
\tilde{f}(0)+c_4^p\varepsilon.
\end{equation}
Now the proof of Lemma~\ref{localmax} can be completed along the argument above
by introducing the obvious alterations.
\end{proof}

\begin{coro}
\label{globalmax} There exists a positive absolute constant $\tilde{c}$ such
that if $p \geq 1$, and  $A(K)=(1+\varepsilon)A(T_K)$ for a convex disc $K$,
then
$$
\E^p_3(K)\leq(1-\tilde{c}^p\varepsilon^2)\E^p_3(T^2)
\mbox{ \ and \ }\E^p_*(K)\leq(1-\tilde{c}^p\varepsilon^2)\E^p_*(T^2).
$$
\end{coro}
\begin{proof} We present the argument only for $\E^p_3(K)$.
Let $\hat{c}$ and  $\varepsilon_0$ come from Lemma~\ref{localmax}. We may
assume that $K$ is an $m$-gon for $m\geq 4$ by continuity, and that
$A(K)>(1+\varepsilon_0)A(T_K)$ by Lemma~\ref{localmax}. It follows by
\eqref{basicdef}, that there exist $m-3$ consecutive basic systems that induce
a continuous deformation of $K$ into a triangle in a way such that
$\E^p_3(\cdot)$ is strictly increasing during the deformation. Therefore there
exists a polygon $K'$ such that $\E^p_3(K')>\E^p_3(K)$, and
$A(K')=(1+\varepsilon_0)A(T_{K'})$. Now we apply Lemma~\ref{localmax} to $K'$,
and using $\varepsilon<3$, we deduce
\[
\E^p_3(K)<\E^p_3(K')\leq(1-\hat{c}^p\varepsilon_0^2)\E^p_3(T^2)<
\mbox{$(1-\frac{\hat{c}^p\varepsilon_0^2}9\cdot\varepsilon^2)\E^p_3(T^2)$}.
\qedhere
\]
\end{proof}
Having Corollary~\ref{globalmax}, Theorem~\ref{triangleoptstab} is a
consequence of the following.

\begin{lemma}
\label{areaBM} If $\delta_{BM}(K,T^2)=1+\delta$ for  a convex
disc $K$, then
$$
(1+\delta) A(T_K)\leq A(K)< (1+\delta)^2 A(T_K).
$$
\end{lemma}

\begin{proof} The upper bound is consequence of the fact that by the definition of the
Banach-Mazur distance,
 there exists a triangle $T'\subset K$, and $x\in T'$,
such that $K\subset (1+\delta)(T'-x)+x$. For the lower bound,
we may assume that $T_K$ is a regular triangle of edge length one.
Let $p_1,p_2,p_3$ be the vertices of $T_K$, and let
$q_1,q_2,q_3$ be the such that $p_i$ is the midpoint of $[q_j,q_k]$,
$\{i,j,k\}=\{1,2,3\}$.
If
$\{i,j,k\}=\{1,2,3\}$, then let $t_i$ be the maximal distance of points of
$K\cap[q_i,p_j,p_k]$ from $[p_j,p_k]$. On the one hand,
\[
A(K)\geq A(T_K)+(t_1+t_2+t_3)/2= \mbox{$(1+\frac2{\sqrt{3}}(t_1+t_2+t_3))$}A(T_K).
\]
On the other hand, $K$ is contained in a regular triangle, that is similarly
situated to $T_K$, and whose height is $\frac{\sqrt{3}}2+t_1+t_2+t_3$. It follows
that
\[
1+\delta\leq \mbox{$(\frac{\sqrt{3}}2+t_1+t_2+t_3)/\frac{\sqrt{3}}2 =
1+\frac2{\sqrt{3}}(t_1+t_2+t_3)$}\leq A(K)/A(T_K). \qedhere
\]
\end{proof}
That the exponent $2$ in the error term $\delta^2$ is optimal is shown by the
example of the closure of $T\backslash \delta T$, where $T$ is a triangle such
that $o$ is a vertex.

\section{Stability of Petty projection inequality}\label{pettystab}

Theorem~\ref{ellipsoid} readily implies the stability version of the
Busemann-Petty centroid inequality \eqref{busepetty}, using
\eqref{centroidbody}. Here we also derive the stability version of Petty's
projection inequality (cf. \cite{Lut93}). Given a convex body $K$, its {\em
projection body} $\Pi K$ is defined by its support function
\[
h_{\Pi K}(u) = V_{d-1}(p_u(K)).
\]
 The Petty
projection inequality states that the quantity
\[
V_d(K)^{d-1} V_d (\Pi^*(K))
\]
is maximised for ellipsoids. Citing formula (5.7) of \cite{Lut93} and using
\eqref{centroidbody}, we arrive to that if $V_d(K)=1$, then
\begin{equation}\label{projcentr}
\frac{1}{V_d(K)^{d-1}V_d(\Pi^* K)} \geq \left(\frac {d+1}{2}\right)^d \frac{V_d
(\Gamma (\Pi^*K))}{V_d (\Pi^* K)} = (d+1)^d \E^1_o (\Pi^*K).
\end{equation}
Let $\delta_{BM}(K, B^d) = 1 + \delta$. Bourgain and Lindenstrauss \cite{BoL88}
proved that there exists a constant $C$ depending on $d$, so that
\[
\delta_{BM}(\Pi K, B^d) \geq 1+ C \delta^{(d^2 + 5d)/2}.
\]
Referring to  $\delta_{BM}(\Pi K, B^d) = \delta_{BM}(\Pi^* K, B^d)$,
Theorem~\ref{ellipsoid} implies that there exists a constant $c$ depending on
$d$ only, so that
\[
\E^1_o (\Pi^*K) \geq (1 + C' \delta^{d(d+3)(d+5)/2})\E^1_o(B^d).
\]
Thus, from \eqref{projcentr} we obtain  that
\[
V_d(K)^{d-1}V_d(\Pi^* K) \leq (1+c \delta^{d(d+3)(d+5)/2})^{-1}
V_d(B^d)^{d-1}V_d(\Pi^* B^d).
\]

We note that the stability version of the Busemann intersection inequality
\eqref{busemanninter} would also follow by verifying a statement of the
following type. If $K$ is a convex body in $\R^d$, and $\delta_{BM} (K, B^d) =
1 + \delta$ for some $\delta>0$, then there exist $\nu, \eta >0$ (depending on
$\delta$) so that $\delta_{BM} (K \cap u^\perp, B^{d-1}) > 1+ \eta$ for a set
of directions $u$ of measure at least~$\nu$. The enthusiast would believe in
such a statement with an absolute constant $\nu$ and $\eta = \delta^q$ for some
$q
>0$.

\section{Acknowledgements}

We would like to thank Imre B\'ar\'any for drawing our attention to the problem and
for useful comments, and Ferenc Fodor, Vitali Milman and Alain Pajor for the
valuable suggestions and enlightening discussions.

\end{document}